\documentclass[11pt]{amsart}

\usepackage[T1]{fontenc}
\usepackage[latin1]{inputenc}
\usepackage[english]{babel}
\usepackage{amsmath,amssymb,amsthm} %amssymb loads amsfonts
\usepackage{tikz}
\usetikzlibrary{matrix,arrows}
\usepackage{tikz-cd}
\usepackage[shortlabels]{enumitem}
\usepackage{mymacros}
\usepackage{hyperref}

\theoremstyle{definition}
\newtheorem*{ack}{Acknowledgements}

\author{Michael Hartz}
\title{von Neumann's inequality for commuting weighted shifts}
\address{Department of Pure Mathematics, University of Waterloo, Waterloo, ON N2L 3G1, Canada}
\email{mphartz@uwaterloo.ca}
\thanks{The author is partially supported by an Ontario Trillium Scholarship.}

\renewcommand{\MR}[1]{}

\begin{document}

\begin{abstract}
  We show that every multivariable contractive weighted shift dilates to a tuple of commuting
  unitaries, and hence satisfies von Neumann's inequality.
  This answers a question of Lubin and Shields.
  We also exhibit a closely related $3$-tuple of commuting contractions, similar to Parrott's example,
  which does not dilate to a $3$-tuple of commuting unitaries.
\end{abstract}

  \maketitle
  \section{Introduction}
  Von Neumann's inequality states that
  \begin{equation*}
    ||p(T)|| \le \sup \{|p(z)|: z \in \ol{\bD} \}
  \end{equation*}
  holds for every contraction $T$ on a Hilbert space and every polynomial $p \in \bC[z]$,
  where $\bD$ denotes the open unit disc in $\bC$ \cite{Neumann51}. This inequality can be deduced from Sz.-Nagy's dilation theorem,
  according to which every contraction $T$ on a Hilbert space admits a unitary (power) dilation \cite{Sz.-Nagy53}.
  And\^o's theorem shows that any pair $(T_1,T_2)$ of commuting contractions dilates to a
  pair of commuting unitaries \cite{Ando63}. As a consequence, we obtain a two variable von Neumann inequality:
  \begin{equation*}
    ||p(T_1,T_2)|| \le \sup \{|p(z_1,z_2)|: (z_1,z_2) \in \ol{\bD}^2 \}
  \end{equation*}
  for every polynomial $p \in \bC[z_1,z_2]$.
  The situation for three or more commuting contractions is quite different. Parrott \cite{Parrott70} gave
  an example of three commuting contractions satisfying von Neumann's inequality
  which do not dilate to commuting unitaries. Kaijser-Varopoulos \cite{Varopoulos74}
  and Crabb-Davie \cite{CD75} exhibited three commuting contractions which do not satisfy the
  three variable version of von Neumann's inequality.
  More details about this topic can be found in Chapter 5 of the book \cite{Paulsen02}.

  In 1974, Shields \cite[Question 36]{Shields74} asked if von Neumann's inequality holds for a particularly tractable
  class of commuting contractions, namely multivariable weighted shifts. He attributes this question to Lubin. This problem
  is also explicitly mentioned in the proof of Theorem 22 in \cite{JL79}.
  Multivariable weighted shifts can be defined as follows.
  Let $(\beta_I)_{I \in \bN^d}$ be a collection of strictly positive numbers with $\beta_0 = 1$ such that
  for $j = 1,\ldots,d$, the set
  \begin{equation*}
    \{\beta_{I + \varepsilon_j} / \beta_I : I \in \bN^d\}
  \end{equation*}
  is bounded, where $\varepsilon_j \in \bN^d$ is the tuple whose $j$-coordinate is $1$
  and whose other coordinates are $0$. Define a space of formal power series
  \begin{equation*}
    H^2(\beta) = \Big\{ f(z) = \sum_{I \in \bN^d} a_I z^I : ||f||^2 = \sum_{I \in \bN^d} |a_I|^2 \beta_I^2 < \infty \Big\}
  \end{equation*}
  and for $j = 1,\ldots,d$, let $M_{z_j}$ be the unique bounded linear operator on $H^2(\beta)$ such that
  \begin{equation*}
    M_{z_j} z^I = z^{I + \varepsilon_j} \quad \text{ for all } I \in \bN^d.
  \end{equation*}
  Then the tuple $(M_{z_1},\ldots,M_{z_d})$ is called a $d$-variable weighted shift. More details
  about multivariable weighted shifts can be found in Section \ref{sec:weighted_multshift}.
  
  The purpose of this note is to provide a positive answer to the question of Lubin and Shields.

  \begin{thm}
    \label{thm:main}
    Let $T = (T_1,\ldots,T_d)$ be a $d$-variable weighted shift and assume that each $T_j$ is a contraction.
    Then $T$ dilates to a $d$-tuple of commuting unitaries. In particular,
    $T$ satisfies von Neumann's inequality, that is,
    \begin{equation*}
      ||p(T)|| \le \sup \{ |p(z)|: z \in \ol{\bD}^d\}
    \end{equation*}
    for all $p \in \bC[z_1,\ldots,z_d]$.
  \end{thm}

  A proof of Theorem \ref{thm:main} will be given in Section \ref{sec:proof}. In fact, we will show that every contractive
  $d$-variable weighted shift satisfies the matrix version of von Neumann's inequality.

  It is important that the tuple $T$ in Theorem \ref{thm:main} is a multivariable weighted shift in the sense described above.
  Indeed, the three operators of the Crabb-Davie example \cite{CD75}, which do not satisfy von Neumann's inequality,
  commute and are weighted shifts individually
  (with some weights equal to zero), but they do not form a $3$-variable weighted shift. Furthermore, it is also
  possible to define multivariable weighted shifts with possibly zero weights, see Remark \ref{rem:zero_weights}.
  In Section \ref{sec:counterexample}, we will exhibit such a tuple of operators which does not dilate
  to a tuple of commuting unitaries. This example is similar to Parrott's example \cite{Parrott70}.

  Abstract considerations show that there exists a $d$-tuple of commuting contractions
  $(S_1,\ldots,S_d)$ on a Hilbert space such that
  \begin{equation*}
    ||p(T_1,\ldots,T_d)|| \le ||p(S_1,\ldots,S_d)||
  \end{equation*}
  holds for every $d$-tuple of commuting contractions $(T_1,\ldots,T_d)$ and every
  polynomial $p \in \bC[z_1,\ldots,z_d]$, and, in fact, for every matrix of polynomials $p$.
  This defines an operator algebra structure on $\bC[z_1,\ldots,z_d]$, which is called the universal operator
  algebra for $d$ commuting contractions (see \cite[Chapter 5]{Paulsen02}). It follows from Theorem \ref{thm:main} and from
  the failure of von Neumann's inequality for three commuting contractions that $S$ cannot be a $d$-variable weighted shift for $d \ge 3$.

  \begin{cor}
    Let $\rho: \bC[z_1,\ldots,z_d] \to \cB(\cH)$ be an isometric representation of the
    universal operator algebra for $d$ commuting contractions.
    If $d \ge 3$, then the tuple $(\rho(z_1),\ldots,\rho(z_d))$
    is not a $d$-variable weighted shift. \qed
  \end{cor}

  This should be compared with the situation for commuting row contractions,
  that is, commuting tuples $(T_1,\ldots,T_d)$ satisfying $\sum_{j=1}^d T_j T_j^* \le 1$.
  In this case, the universal norm is the multiplier norm on the Drury-Arveson space, and the
  corresponding $d$-tuple $(M_{z_1},\ldots,M_{z_d})$ of row contractions is a $d$-variable weighted shift \cite{Arveson98,Drury78}.
  Indeed, the tuple $(M_{z_1},\ldots,M_{z_d})$ was first described as a weighted shift.

  The remainder of this note is organized as follows. In Section \ref{sec:von_neumann_shilov}, we provide a general method
  for establishing von Neumann's inequality for commuting contractions. In Section \ref{sec:weighted_multshift}, we recall
  the definition and some basic properties of multivariable weighted shifts. Section \ref{sec:proof} contains
  the proof of Theorem \ref{thm:main}. Finally, in Section \ref{sec:counterexample}, we exhibit an example which shows
  that Theorem \ref{thm:main} does not generalize to multivariable weighted shifts with possibly zero weights.

  \section{A general method for establishing von Neumann's inequality}
  \label{sec:von_neumann_shilov}

  Let $X \subset \bC^N$ be a compact set. We say that a function $f: X \to \bC$ is \emph{analytic}
  if it extends to an analytic function in an open neighbourhood of $X$.
  We denote by $\partial_0 X$ the Shilov boundary
  of the algebra of all analytic functions on $X$.
  Thus, $\partial_0 X$ is the
  smallest compact subset $K$ of $X$ such that
  \begin{equation*}
    \sup \{ |f(z)|: z \in X \} = \sup \{ |f(z)| : z \in K \}
  \end{equation*}
  holds for every analytic function $f$ on $X$.
  For simplicity, we call $\partial_0 X$ the \emph{Shilov boundary of $X$.}
  By the maximum modulus principle, $\partial_0 X$ is contained in the topological
  boundary $\partial X$, but it may be smaller.
  Much as in the scalar valued case,
  we say that a function $F = (F_1,\ldots,F_d): X \to \cB(\cH)^d$ is analytic if each
  $F_j$ extends to a $\cB(\cH)$-valued analytic function in an open neighbourhood of $X$.

  The next result is motivated by a proof of von Neumann's inequality for matrices due to Nelson \cite{Nelson61},
  see also \cite[Exercise 2.16]{Paulsen02} and \cite[Chapter 1]{Pisier01}.

  \begin{prop}
    \label{prop:shilov_von_neumann}
    Let $X \subset \bC^N$ be compact and suppose that $T: X \to \cB(\cH)^d$
    is an analytic function
    such that $T(z)$ is a $d$-tuple of commuting contractions for all $z \in X$. Then the following
    assertions are true:
    \begin{enumerate}[label=\normalfont{(\alph*)}]
      \item If the tuple $T(z)$ satisfies von Neumann's inequality for all $z \in \partial_0 X$,
        then $T(z)$ satisfies von Neumann's inequality for all $z \in X$.
      \item If the tuple $T(z)$ dilates to a tuple of commuting unitaries for all $z \in \partial_0 X$,
        then $T(z)$ dilates to a tuple of commuting unitaries for all $z \in X$.
    \end{enumerate}
  \end{prop}

  \begin{proof}
    Let $p = (p_{i,j})_{i,j=1}^n$ be an $n \times n$ matrix of polynomials in $\bC[z_1,\ldots,z_d]$ and suppose
    that the inequality
    \begin{equation*}
      ||p(T(z))||_{\cB(\cH^n)} \le ||p||_\infty
    \end{equation*}
    holds for all $z \in \partial_0 X$, where
    \begin{align*}
      ||p||_\infty =
      \sup \{ ||p(w)||_{M_n} : w \in \ol{\bD}^d \}.
    \end{align*}
    Given $f,g \in \cH^n$ of norm $1$, observe that the scalar valued function
    \begin{equation*}
      X \to \bC, \quad z \mapsto \langle p(T(z)) f, g \rangle,
    \end{equation*}
    is analytic. By assumption, this function is bounded by $||p||_\infty$
    on $\partial_0 X$, and hence on $X$ by definition of $\partial_0 X$.
    Consequently, the inequality
    \begin{equation*}
      ||p(T(z))||_{\cB(\cH^n)} \le ||p||_\infty
    \end{equation*}
    holds for all $z \in X$.
    Part (a) now follows by taking $n=1$ above. Part (b) is a consequence of the general fact
    that a tuple of commuting contractions satisfies the matrix version of von Neumann's inequality if
    and only if it dilates to a tuple of commuting unitaries, which follows from Arveson's dilation theorem (see,
    for example, \cite[Corollary 7.7]{Paulsen02}, or \cite[Corollary 4.9]{Pisier01} for the explicit statement).
  \end{proof}

  \begin{rem}
    \begin{enumerate}[(a),wide]
      \item
    Proposition \ref{prop:shilov_von_neumann} and its proof remain valid in the following more general setting: Suppose
    that $\cA \subset C(X)$ is a uniform algebra with Shilov boundary $X_0 \subset X$. Let $T: X \to \cB(\cH)^d$
    be a function such that $T(z)$ is a $d$-tuple of commuting contractions for all $z \in X$ and such that
    for all $p \in \bC[z_1,\ldots,z_d]$ and all $f,g \in \cH$, the scalar valued function
    \begin{equation*}
      X \to \bC, \quad z \mapsto \langle p(T(z)) f ,g \rangle,
    \end{equation*}
    belongs to the algebra $\cA$. If $T(z)$ satisfies von Neumann's inequality (respectively dilates
    to a $d$-tuple of commuting unitaries) for all $z \in X_0$, then $T(z)$ satisfies von Neumann's inequality
    (respectively dilates to a $d$-tuple of commuting unitaries) for all $z \in X$.
  \item
    We can recover Nelson's proof of von Neumann's inequality from Proposition \ref{prop:shilov_von_neumann}
    in the following way.
    Suppose that $T \in M_n(\bC)$ is a contraction
    and let $T = U D V$ be a singular value decomposition of $T$, where $U,V \in M_n(\bC)$
    are unitary and $D$ is a diagonal matrix with entries in $[0,1]$. For $z \in \overline{\bD}^d$, define
    \begin{equation*}
      T(z) = U \operatorname{diag}(z_1,\ldots,z_n) V.
    \end{equation*}
    This defines an analytic map on $\ol{\bD}^d$.
    Moreover,
    $\partial_0 (\overline{\bD}^d) = \bT^d$, and $T(z)$ is unitary for $z \in \bT^d$.
    By Proposition \ref{prop:shilov_von_neumann}, it therefore suffices to establish
    von Neumann's inequality for unitary matrices, which in turn is an immediate consequence
    of the spectral theorem.
    \end{enumerate}

  \end{rem}

  To motivate the proof of Theorem \ref{thm:main}, we first deduce from Proposition \ref{prop:shilov_von_neumann}
  that
  single contractive weighted shifts satisfy von Neumann's inequality (of course, this also follows
  from the usual von Neumann's inequality for Hilbert space contractions).

  \begin{prop}
    \label{prop:single_variable}
    Let $T$ be a unilateral weighted shift which is a contraction. Then $T$ satisfies
    von Neumann's inequality.
  \end{prop}

  \begin{proof}
    A straightforward approximation argument reduces the statement to the case
    of truncated weighted shifts (see Lemma \ref{lem:truncation} below for the details).
    Let $n \in \bN$ and suppose that
    $T \in M_n(\bC)$ is a truncated weighted shift
    with weight sequence $w_1,\ldots,w_{n-1}$ in $\ol{\bD}$, that is,
    \begin{equation*}
      T =
      \begin{pmatrix}
        0 & 0 & \cdots & 0 & 0\\
        w_1 & 0 & \cdots & 0 & 0 \\
        0 & w_2 & \cdots & 0 & 0 \\
        \vdots & \vdots & \ddots & \vdots & \vdots \\
        0 & 0 & \ldots & w_{n-1} & 0
      \end{pmatrix}.
    \end{equation*}
    For $z = (z_1,\ldots,z_{n-1}) \in \ol{\bD}^{n-1}$, let $T(z) \in M_n$ be the truncated weighted
    shift with weight sequence $z_1,\ldots,z_{n-1}$. This defines an analytic map on $\ol{\bD}^{n-1}$.
    Since $\partial_0(\ol{\bD}^{n-1}) = \bT^{n-1}$, an application of Proposition
    \ref{prop:shilov_von_neumann} shows that it suffices to establish
    von Neumann's inequality for $T(z)$ if
    $z \in \bT^{n-1}$. However, for $z \in \bT^{n-1}$, the operator $T(z)$
    is easily seen to be unitarily equivalent to $T(1,1,\ldots,1)$ (cf. Corollary 1 in Section 2 of \cite{Shields74}),
    which evidently dilates to the bilateral shift, and thus satisfies von Neumann's inequality.
  \end{proof}

  \section{Preliminaries about weighted shifts}
  \label{sec:weighted_multshift}

  In this section, we review the definition and some basic properties of multivariable weighted shifts.
  For a comprehensive treatment, the reader is referred to \cite{JL79}. Let $d \in \bN$.
  We begin by recalling multi-index notation. A multi-index is an element $I \in \bN^d$. For
  $1 \le j \le d$, we write $\varepsilon_j$ for the multi-index $I = (i_1,\ldots,i_d)$ with
  $i_j = 1$ and $i_k =0$ for $k \neq j$. Given a multi-index $I = (i_1,\ldots,i_d)$, we define
  \begin{equation*}
    |I| = i_1 + \ldots + i_d.
  \end{equation*}
  Moreover, if $z = (z_1,\ldots,z_d) \in \bC^d$, we write
  \begin{equation*}
    z^I = z_1^{i_1} \ldots z_d^{i_d}.
  \end{equation*}
  If $T = (T_1,\ldots,T_d)$ is a commuting tuple of operators, we similarly define $T^I$.
  Given two multi-indices $I = (i_1,\ldots,i_d)$ and $J= (j_1,\ldots,j_d)$, we say that $I \le J$ if
  $i_k \le j_k$ for $1 \le k \le d$.

  Now, let $\cH$ be a Hilbert space with an orthonormal basis
  \begin{equation*}
    \{e_I: I \in \bN^d\}
  \end{equation*}
  and let
  \begin{equation*}
    \mathbf{w} = (w_{I,j})_{(I,j) \in \bN^d \times \{1,\ldots,d\}}
  \end{equation*}
  be a bounded collection of strictly positive numbers satisfying the commutation relations
  \begin{equation}
    \label{eqn:comm}
    w_{I,j} w_{I + \varepsilon_j, k} = w_{I,k} w_{I + \varepsilon_k,j}
  \end{equation}
  for all $I \in \bN^d$ and $j \in \{1,\ldots,d\}$.
  The \emph{($d$-variable) weighted shift} with weights $\mathbf w$
  is the unique $d$-tuple of bounded operators $(T_1,\ldots,T_d)$ on $\cH$ satisfying
  \begin{equation*}
    T_j e_I = w_{I,j} e_{I + \varepsilon_j} \quad (I \in \bN^d, j \in \{1,\ldots,d\}).
  \end{equation*}
  Observe that the relations \eqref{eqn:comm} guarantee that the operators $T_j$ commute.
  Evidently, $T_j$ is a contraction if and only if $w_{I,j} \le 1$ for all $I \in \bN^d$.

  \begin{rem}
    \begin{enumerate}[(a),wide,ref={\therem~(\alph*)}]
    \item
    \label{rem:H2_beta}
    The definition of multivariable weighted shifts in the introduction is equivalent to the definition given in this section.
    To see this, suppose that $M_z = (M_{z_1},\ldots,M_{z_d})$ is a tuple of multiplication operators on a space
    $H^2(\beta)$ as in the introduction. Then with respect to the orthonormal basis $(z^I / \beta_I)_{I \in \bN^d}$,
    the tuple $M_z$ is the $d$-variable weighted shift with weights
    \begin{equation*}
      w_{I,j} = \frac{\beta_{I + \varepsilon_j}}{\beta_I}.
    \end{equation*}
    Conversely, every $d$-variable weighted shift in the sense of this section is unitarily equivalent to the
    tuple $(M_{z_1},\ldots,M_{z_d})$ on $H^2(\beta)$, where
    \begin{equation*}
      \beta_I = ||T^I e_{(0,\ldots,0)}||
    \end{equation*}
    for $I \in \bN^d$, see \cite[Proposition 8]{JL79}. While the definition in terms of multiplication
    operators is somewhat cleaner, it is more convenient to work with the weights $\mathbf{w}$
    and not with the weights $\beta$ in the proof of Theorem \ref{thm:main}.
    Nevertheless, part of the proof is motivated by this other point of view.
    \item \label{rem:zero_weights} The assumption that all weights $w_{I,j}$ are strictly
      positive is standard in the study of multivariable weighted shifts. It is of course possible
      to define multivariable weighted shifts with complex weights in a similar way.
      According to \cite[Corollary 2]{JL79},
      every multivariable weighted shift with complex non-zero weights, equivalently, every injective
      multivariable weighted shift with complex weights, is unitarily equivalent to one with strictly positive weights.
      However, if we allow for zero weights, then the situation is quite different. Such a tuple is no longer unitarily
      equivalent to a tuple of the form $M_z$ on $H^2(\beta)$.
      We will see in Section 5 that the dilation part of Theorem \ref{thm:main} does not hold in this more general setting.
    \end{enumerate}
  \end{rem}

  Just as in the proof of Proposition \ref{prop:single_variable}, we will
  work with truncated shifts in the proof of Theorem \ref{thm:main}, and we will also need to consider complex and possibly
  zero weights.
  For $N \in \bN$, define
  a finite dimensional subspace $\cH_N$ of $\cH$ by
  \begin{equation}
    \label{eqn:H_N}
    \cH_N = \operatorname{span} \{e_I: |I| \le N \}.
  \end{equation}
  Suppose that
  \begin{equation*}
    \mathbf{w} = (w_{I,j})_{|I| \le N,  j \in \{1,\ldots,d\}}
  \end{equation*}
  is a collection of complex numbers satisfying the commutation relations \eqref{eqn:comm}
  for $|I| \le N-1$ and $j \in \{1,\ldots,d\}$. We call such a collection a \emph{commuting family}.
  The \emph{($d$-variable) truncated weighted shift} with weights $\mathbf{w}$
  is the unique $d$-tuple of operators $(T_1,\ldots,T_d)$ on $\cH_{N+1}$ satisfying
  \begin{equation*}
    T_j e_I =
    \begin{cases}
      w_{I,j} e_{I + \varepsilon_j} & \text{ if } |I| \le N, \\
      0 & \text{ if } |I| = N+1.
    \end{cases}
  \end{equation*}
  Once again, the commutation relations
  ensure that the operators $T_j$ commute.

  We require the following straightforward adaptation of \cite[Corollary 2]{JL79} to truncated weighted shifts.

  \begin{lem}
    \label{lem:weights_unit}
    Let $T$ be a $d$-variable truncated weighted shift on $\cH_{N+1}$ with non-zero weights $\mathbf{w} = (w_{I,j})$.
    Then $T$ is unitarily equivalent to the $d$-variable truncated weighted shift with weights $(|w_{I,j}|)$.
  \end{lem}

  \begin{proof}
    For $|I| \le N+1$, we define recursively complex
    numbers $\lambda_I$ of modulus $1$ by $\lambda_{(0,\ldots,0)} = 1$ and
    $\lambda_{I + \varepsilon_j} = \lambda_I w_{I,j}/ |w_{I,j}|$ for $|I| \le N$.
    As in the proof of \cite[Corollary 2]{JL79}, we deduce from the commutation relations \eqref{eqn:comm}
    that this is well defined. If $U$ is the unitary
    operator on $\cH_{N+1}$ satisfying $U e_I = \lambda_I e_I$,
    then $(U^* T_1 U, \ldots, U^* T_d U)$ is the $d$-variable truncated weighted shift
    with weights $(|w_{I,j}|)$.
  \end{proof}

  \section{Proof of Theorem \ref{thm:main}}
  \label{sec:proof}

  The proof of Theorem \ref{thm:main} is essentially an adaptation
  of the proof of Proposition \ref{prop:single_variable} to the multivariate
  setting.
  We begin with a straightforward reduction to truncated shifts.

  \begin{lem}
    \label{lem:truncation}
    In order to prove Theorem \ref{thm:main}, it suffices to show that
    every $d$-variable truncated weighted shift with weights in $(0,1]$
    dilates to a $d$-tuple of commuting unitaries.
  \end{lem}

  \begin{proof}
    Let
    $T = (T_1,\ldots,T_d)$ be a $d$-variable weighted shift on $\cH$
    such that $||T_j|| \le 1$ for each $j$, that is, all weights of $T$ belong to $(0,1]$.
    Let $\cH_N$
    be the subspace of $\cH$ defined in \eqref{eqn:H_N}.
    Observe that the compressed tuple
    \begin{equation*}
      P_{\cH_N} T \big|_{\cH_N} = (P_{\cH_N} T_1 \big|_{\cH_N},\ldots, P_{\cH_N} T_d \big|_{\cH_N})
    \end{equation*}
    is a $d$-variable truncated weighted shift with weights in $(0,1]$. Since $\cH_N$ is co-invariant under each $T_j$,
    we see that
    \begin{equation*}
      p(P_{\cH_N} T \big|_{\cH_N}) = P_{\cH_N} p(T) \big|_{\cH_N}
    \end{equation*}
    holds for every $p \in \bC[z_1,\ldots,z_d]$. Hence, for every $p \in \bC[z_1,\ldots,z_d]$,
    the sequence  $p(P_{\cH_N} T \big|_{\cH_N})$ converges to $p(T)$ in the strong operator
    topology as $n \to \infty$.
    Consequently, if $P_{\cH_N} T \big|_{\cH_N}$ dilates to a $d$-tuple of commuting unitaries and thus
    satisfies the matrix version of von Neumann's inequality for all $N \in \bN$, then
    $T$ satisfies the matrix version of von Neumann's inequality, and therefore dilates to
    a $d$-tuple of commuting unitaries.
  \end{proof}

  The main obstacle when generalizing the proof of Proposition \ref{prop:single_variable} to multivariable shifts
  is that multivariable truncated weighted shifts are not
  parametrized by the points of a polydisc in an obvious way. This is because of
  the commutation relations \eqref{eqn:comm}.  Instead, we will use Lemma \ref{lem:Shilov_weights} below.

  For the remainder of this section, let us fix $N \in \bN$ and set
  \begin{equation*}
    \cI = \{(I,j) \in \bN^d \times \{1,\ldots,d\}: |I| \le N\}.
  \end{equation*}
  Let
  $X$ denote the closure of the set of all commuting families $(w_{I,j})_{(I,j) \in \cI}$ with
  $0 < |w_{I,j}| \le 1$ for all $(I,j) \in \cI$.
  Observe that we may regard $X$ as a compact subset of $\bC^{|\cI|}$.

  \begin{lem}
    \label{lem:Shilov_weights}
    The Shilov boundary of $X$ is contained in the set
    \begin{equation*}
      X_0 = \{(w_{I,j}) \in X : |w_{I ,j}| =1 \text{ for all } (I,j) \in \cI \}.
    \end{equation*}
  \end{lem}

  \begin{proof}
    Let $\mathbf{w} = (w_{I,j}) \in X \setminus X_0$ with $w_{I,j} \neq 0$ for all $(I,j) \in \cI$
    and let $f: X \to \bC$ be a function which extends to
    be analytic in a neighbourhood of $X$. We will show that there exists a point
    $\mathbf{\widetilde w} = (\widetilde w_{I,j}) \in X$ with $\widetilde w_{I,j} \neq 0$ for all $(I,j) \in \cI$ such that
    \begin{equation*}
      |f(\mathbf{w})| \le |f(\mathbf{\widetilde w})|
    \end{equation*}
    and such that
    \begin{equation*}
      \{ (I,j) \in \cI: |w_{I,j}| = 1 \} \subsetneq
      \{ (I,j) \in \cI: |\widetilde w_{I,j}| = 1 \}.
    \end{equation*}
    Once this has been accomplished, iterating this process finitely many times yields
    a point $\mathbf{v} \in X_0$ such that $|f(\mathbf w)| \le |f(\mathbf{v})|$.
    Consequently, $X_0$ is a boundary for the algebra of all analytic functions on $X$, so $\partial_0 X \subset X_0$.

    Let us begin by establishing some terminology which will
    be used throughout the proof. Let $T = (T_1,\ldots,T_d)$ be the $d$-variable truncated weighted shift
    with weights $\mathbf{w}$. If $I \subset \bN^d$ is a multi-index with $|I| \le N+1$,
    we say that $I$ is \emph{good} if
    \begin{equation*}
      ||T^I e_{(0,\ldots,0)} || = 1.
    \end{equation*}
    Otherwise, we call $I$ \emph{bad} (cf. Remark \ref{rem:H2_beta}).
    The following observations are immediate:
    \begin{enumerate}[(a)]
      \item If $I$ is good and if $J \le I$, then $J$ is good.
      \item If $(I,j) \in \cI$ with $|w_{I,j}| < 1$, then $I+\varepsilon_j$ is bad.
      \item Suppose that $|I| \le N$. If $I$ is good and $I + \varepsilon_j$ is bad,
        then $|w_{I,j}| < 1$.
    \end{enumerate}
    We say that a pair $(I,j) \in \cI$ is \emph{scalable} if $I$ is good, but $I + \varepsilon_j$ is bad.
    It follows from (b) and the choice of $\mathbf{w}$ that there exists at least one bad
    multi-index. Since $(0,\ldots,0)$ is good, we therefore see that
    there exists at least one scalable pair.
    Recall that all $|w_{I,j}|$ are assumed to be non-zero, so we may define
    \begin{equation*}
      r = \max \{|w_{I,j}|: (I,j) \text{ is scalable} \}^{-1}.
    \end{equation*}
    Then $r  > 1$ by (c).
    Let $\ol{D_r(0)} \subset \bC$ denote the closed disc of radius $r$ around $0$.
    For $t \in \ol{D_r(0)}$ and $(I,j) \in \cI$, define
    \begin{equation*}
      \widehat{w}_{I,j} (t) =
      \begin{cases}
        t w_{I,j} & \text{ if } (I,j) \text{ is scalable}, \\
        w_{I,j} & \text{ if } (I,j) \text{ otherwise}
      \end{cases}
    \end{equation*}
    and let $\mathbf{\widehat w}(t) = (\widehat w_{I, j}(t))_{(I,j) \in \cI}$.
    We finish the proof by showing that $\mathbf{\widehat w}(t) \in X$ for every $t \in \ol{D_r(0)}$.
    Indeed, it then follows from
    the maximum modulus principle that there exists $t_0 \in \partial D_r(0)$ with
    \begin{equation*}
      |f(\mathbf{w})| = |f(\mathbf {\widehat w}(1))| \le
      |f(\mathbf {\widehat w}(t_0))|,
    \end{equation*}
    so setting $\mathbf{\widetilde w} = \mathbf{\widehat w}(t_0)$, we obtain a point with
    the desired properties.

    Since $X$ is closed, it suffices to show that $\mathbf{\widehat w}(t) \in X$ for all $t \in \ol{D_r(0)} \setminus \{0\}$.
    Clearly, $0 < |\widehat w_{I,j}(t)| \le 1$ for these $t$, so we need to show that $\mathbf{\widehat w}(t)$
    is a commuting family, that is, we need to show that
    \begin{equation*}
      \widehat w_{I,j}(t) \widehat w_{I + \varepsilon_j,k}(t) =
      \widehat w_{I,k}(t) \widehat w_{I + \varepsilon_k,j}(t)
    \end{equation*}
    for all $t \in \ol{D_r(0)}$ and
    all multi-indices $I$ with $|I| \le N-1$ and $1 \le j,k \le d$.
    Let $I$ be such a multi-index.
    If $I$ is bad, it follows from (a) that $I + \varepsilon_j$ and $I + \varepsilon_k$ are bad as well,
    and hence no pairs in $\cI$ which appear in the above equation are scalable. If $I$ and $I + \varepsilon_j + \varepsilon_k$
    are good, then it follows again from (a) that no pairs in the equation are scalable. Thus, it remains
    to consider the case where $I$ is good and $I + \varepsilon_j + \varepsilon_k$ is bad. In this
    case, exactly one of $(I,j)$ and $(I + \varepsilon_j, k)$ is scalable, depending on whether $I + \varepsilon_j$ is good.
    Similarly, exactly one of $(I,k)$ and $(I+\varepsilon_k,j)$ is scalable. Thus
    \begin{equation*}
      \widehat w_{I,j} (t) \widehat w_{I + \varepsilon_j,k} (t)=
      t w_{I,j}  w_{I + \varepsilon_j,k} =
      t w_{I,k} w_{I + \varepsilon_k,j} =
      \widehat w_{I,k}(t) \widehat w_{I + \varepsilon_k,j} (t),
    \end{equation*}
    as asserted.
  \end{proof}

  We are now ready to prove the main theorem.
  \begin{proof}[Proof of Theorem \ref{thm:main}]
    According to Lemma \ref{lem:truncation}, it is enough to establish Theorem \ref{thm:main}
    when $T$ is a $d$-variable truncated weighted shift with weights in $(0,1]$. Assume that
    $T$ acts on $\cH_{N+1}$.
    Given a commuting family $\mathbf w \in X$, let us write $T(\mathbf w)$ for the $d$-variable truncated weighted shift on $\cH_{N+1}$
    with weights $\mathbf w$. Then the range of the analytic map
    \begin{equation*}
      X \to \cB(\cH_{N+1})^d, \quad \mathbf w \mapsto T(\mathbf w),
    \end{equation*}
    consists of $d$-tuples of commuting contractions and contains every
  $d$-variable truncated weighted shift on $\cH_{N+1}$ with weights in $(0,1]$.
    According to Proposition \ref{prop:shilov_von_neumann} and Lemma \ref{lem:Shilov_weights},
    it therefore suffices to show that $T(\mathbf w)$ dilates to a $d$-tuple of commuting unitaries
    if $\mathbf{w} \in X_0$. We infer from Lemma \ref{lem:weights_unit} that for $\mathbf{w} \in X_0$, the $d$-tuple
    $T(\mathbf{w})$ is unitarily equivalent to the tuple $T(\mathbf{1})$, where $\mathbf{1}$
    denotes the element of $X_0$ whose entries are all equal to $1$. Thus, it remains to show that $T(\mathbf{1})$
    dilates to a $d$-tuple of commuting unitaries. In this case, it is not hard to construct
    a unitary dilation explicitly.

    Indeed, let $\sigma$ denote the normalized Lebesgue measure
    on $\bT^d$, and for $1 \le k \le d$, let $M_{z_k}$ be the operator
    on $L^2(\sigma)$ given by multiplication with $z_k$. Then $M_z = (M_{z_1},\ldots,M_{z_d})$ is a $d$-tuple
    of commuting unitaries, and $T(\mathbf{1})$ can be identified with the compression
    of $M_z$ to the semi-invariant subspace
    \begin{equation*}
      \operatorname{span} \{z^I: I \in \bN^d, |I| \le N \},
    \end{equation*}
    cf. Example 1 in Section 2 of \cite{JL79}. Therefore, the proof is complete.
  \end{proof}

  \begin{rem}
    In the last proof, the tuple of unitaries $(M_{z_1}^*, \ldots, M_{z_d}^*)$ on $L^2(\bT^d)$ is in fact a regular
    dilation of the adjoint of $T(\mathbf{1})$ in the sense of \cite[Section 9]{SFB+10}. Therefore, the adjoint of
    every tuple $T(\mathbf{w})$ for $\mathbf{w} \in \partial_0 X$ admits a regular unitary dilation.
    This is not true for the adjoint of every tuple $T(\mathbf{w})$ for $\mathbf w \in X$.

    For example,
    suppose that $d = 2$ and let
    \begin{equation*}
      w_{I,j} =
      \begin{cases}
        \frac{1}{2}, & \text{ if } I = (0,0) \\
        1, & \text{ otherwise.}
      \end{cases}
    \end{equation*}
    Let $T = (T_1,T_2)$ be the $2$-variable weighted shift with weights $(w_{I,j})$. With notation
    as in Remark \ref{rem:H2_beta}, $T$ is unitarily equivalent to
    $(M_{z_1},M_{z_2})$ on $H^2(\beta)$, where $\beta_0 = 1$ and $\beta_{I} = \frac{1}{2}$
    if $I \neq (0,0)$.
    A straightforward computation shows that
    \begin{equation*}
      (1 - T_1 T_1^* - T_2 T_2^* + T_1 T_2 T_1^* T_2^*) e_{(1,1)} = - \frac{3}{4 }e_{(1,1)},
    \end{equation*}
    hence $T^*$ does not admit a regular unitary dilation by \cite[Theorem 9.1]{SFB+10}.
    Similarly, the truncations $P_{\cH_N} T^* \big|_{\cH_N}$ do not admit regular
    unitary dilations if $N \ge 2$.

    For the same reason, multivariable truncated weighted shifts do not in general coextend to a (direct sum of) $M_z$ on
    the Hardy space $H^2(\bD^d)$, or, more generally, to a tuple $(V_1,\ldots,V_d)$ of doubly
    commuting isometries (i.e. the $V_i$ commute and $V_i^* V_j = V_j V_i^*$ if $i \neq j$). Indeed,
    if $(V_1,V_2)$ is a pair of doubly commuting isometries, then
    \begin{equation*}
      (1 - V_1 V_1^* - V_2 V_2^* + V_1 V_2 V_1^* V_2^*) = (1 - V_1 V_1^*) (1 - V_2 V_2^*)
    \end{equation*}
    is a positive operator, and hence if $(T_1,T_2)$ is a compression
    of $(V_1,V_2)$ to a co-invariant subspace, then
    \begin{equation*}
      1 -T_1 T_1^* - T_2 T_2^* + T_1 T_2 T_1^* T_2^*
    \end{equation*}
    is a positive operator as well. On the other hand, if $d = 1$, then every contractive weighted
    shift, being a pure contraction, coextends to a direct sum of copies of the unilateral shift $M_z$ on $H^2(\bD)$.

  \end{rem}

  \section{A non-injective counterexample}
  \label{sec:counterexample}
  The definiton of multivariable weighted shifts given in Section \ref{sec:weighted_multshift} can be generalized
  to allow complex and possibly zero weights, see Remark \ref{rem:zero_weights}.
  Even though it is customary to assume that all weights are non-zero, we may still ask if Theorem \ref{thm:main} holds
  in this greater generality. Observe that the proof of Theorem \ref{thm:main}
  breaks down if some of the weights of $T$ are zero. Indeed,
  the method of scaling certain weights by a complex number $t$ never changes a zero weight into
  a non-zero one. It is natural to ask, however, if the proof can be modified by introducing
  non-zero weights in such a way that the commutation relations \eqref{eqn:comm} still hold.
  We will now exhibit an example which shows that this is not possible in general. In fact,
  the operator tuple in this example does not dilate to a commuting tuple of unitaries.

  Let $T$ be a $3$-variable weighted shift with not necessarily positive weights $(w_{I,j})$ given by
  \begin{equation*}
    w_{I,j} =
    \begin{cases}
      0, & \text{ if } |I| \ge 2 \text{ or } I = \varepsilon_j, \\
      a_{i,j}, & \text{ if } I = \varepsilon_i \text{ and } i \neq j, \\
      \delta_j & \text{ if } I = (0,0,0).
    \end{cases}
  \end{equation*}
  Here $(a_{i, j})_{i \neq j}$ are six complex numbers of modulus $1$ and $(\delta_j)_{1 \le j \le 3}$ are three
  complex numbers of modulus at most $1$, all to be determined later.
  The relevant part of the three dimensional grid $\bN^3$ together with the weights
  $w_{I,j}$ is shown below.

  
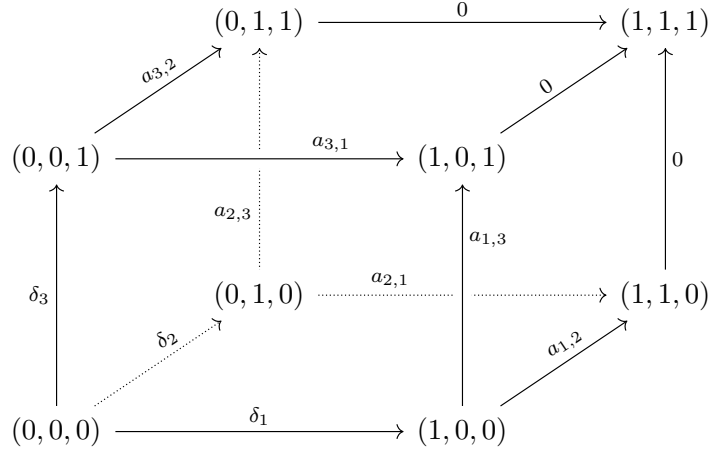
\begin{figure}[h]
  \begin{tikzcd}[back line/.style={densely dotted}, row sep=3em, column sep=3em]
    & (0,1,1) \ar{rr}{0} \ar[back line,leftarrow]{dd}[near end,swap]{a_{2,3}} 
  & & (1,1,1) \\
  (0,0,1) \ar{ur}[sloped, near end]{a_{3,2}}  \ar[crossing over]{rr}[near end]{a_{3,1}} 
  & & (1,0,1) \ar{ur}[sloped]{0}  \\
  & (0,1,0) \ar[back line]{rr}[near start]{a_{2,1}}
  & & (1,1,0) \ar{uu}[swap]{0}  \\
  (0,0,0) \ar{rr}{\delta_1} \ar[back line]{ur}[sloped,near end]{\delta_2} \ar{uu}{\delta_3} & & (1,0,0) \ar{ur}[sloped,near end]{a_{1,2}} \ar[crossing over]{uu}[swap,near end]{a_{1,3}}
\end{tikzcd}
\end{figure}

  Observe that if $\delta_j = 0$ for all $j$, then $(w_{I,j})$ satisfies the commutation relations \eqref{eqn:comm} regardless
  of the value of the six weights $a_{i,j}$. On the other hand, the relations
  \begin{align*}
    \delta_1 a_{1,3} &= \delta_3 a_{3,1} \\
    \delta_2 a_{2,1} &= \delta_1 a_{1,2} \\
    \delta_3 a_{3,2} &=  \delta_2 a_{2,3}
  \end{align*}
  show that if one of the $\delta_j$ is not zero, then all of them are non-zero. Moreover,
  multiplying the above equations, we see that in this case, the equation
  \begin{equation*}
    a_{1,3} a_{2,1} a_{3,2} = a_{3,1} a_{1,2} a_{2,3}
  \end{equation*}
  must hold. For example, let us set $a_{2,1} = -1$ and
  \begin{equation*}
    a_{1,2} = a_{1,3} = a_{2,3} = a_{3,1} =  a_{3,2} = 1.
  \end{equation*}
  If $\delta_j = 0$ for $j = 1,2,3$, then we obtain a $3$-variable contractive weighted shift $T$ with not necessarily positive weights.
  However, it is not possible to perturb the first three weights $w_{(0,0,0),j} = \delta_j$ while maintaining
  commutativity of the operators. Note this also shows that for any $N \ge 2$, the weights $(w_{I,j})$, where $|I| \le N$,
  do not belong to the set $X$ of Section \ref{sec:proof}.

  We now show that the $3$-tuple of commuting contractions $T$ which we just constructed does not
  dilate to a $3$-tuple of commuting unitaries. This is very similar to Parrott's example \cite{Parrott70}.
  Observe that the $6$-dimensional space $M$ spanned by the vectors
  \begin{equation*}
    e_{(1,0,0)}, e_{(0,1,0)},e_{(0,0,1)}, e_{(1,1,0)},e_{(1,0,1)},e_{(0,1,1)}
  \end{equation*}
  contains $\ran(T_j)$ as well as $\ker(T_j)^\bot$ for $j=1,2,3$, so we may restrict
  our attention to this space. With respect to the orthonormal basis above, the operators $T_j$
  are given on $M$ by
  \begin{equation*}
    T_j =
    \begin{bmatrix}
      0 & 0 \\ A_j & 0
    \end{bmatrix} \in M_6(\bC) \quad (j = 1,2,3),
  \end{equation*}
  where
  \begin{equation*}
    A_1 =
    \begin{bmatrix}
      0  & -1 & 0 \\ 0 & 0 & 1 \\ 0 & 0 & 0
    \end{bmatrix}, \quad
    A_2 =
    \begin{bmatrix}
      1  & 0 & 0 \\ 0 & 0 & 0 \\ 0 & 0 & 1
    \end{bmatrix}, \quad
    A_3 =
    \begin{bmatrix}
      0  & 0 & 0 \\ 1 & 0 & 0 \\ 0 & 1 & 0
    \end{bmatrix}.
  \end{equation*}
  As in the treatment of Parrott's example in \cite[Example 20.27]{Davidson88}, we consider the matrix polynomial
  \begin{equation*}
    p(z_1,z_2,z_3) =
    \begin{bmatrix}
      z_1 & z_2 & 0 \\ z_3 & 0 & z_2 \\ 0 & z_3 & -z_1
    \end{bmatrix}.
  \end{equation*}
  It is shown there that
  \begin{equation*}
    \sup \{ ||p(z_1,z_2,z_3)|| : z_1,z_2,z_3 \in \ol{\bD} \} = \sqrt{3}.
  \end{equation*}
  On the other hand,
  \begin{equation*}
    ||p(T_1,T_2,T_3)|| = ||p(A_1,A_2,A_3)||.
  \end{equation*}
  The submatrix of the $9 \times 9$ matrix $p(A_1,A_2,A_3)$ corresponding to
  the rows $1,6,8$ and the columns $2,4,9$ is
  \begin{equation*}
    \begin{bmatrix}
      -1 & 1 & 0 \\ 1 & 0 & 1 \\ 0 & 1 & -1
    \end{bmatrix},
  \end{equation*}
  and it is easy to check that this matrix has norm $2$. Thus,
  \begin{equation*}
    ||p(T_1,T_2,T_3)|| \ge 2.
  \end{equation*}
  In fact, it is not hard to see that $||p(T_1,T_2,T_3)|| = 2$. Since $2 > \sqrt{3}$,
  it follows that the commuting contractions $T_1,T_2,T_3$ do not dilate to commuting unitaries.
  However, it follows from Section 5 of \cite{Parrott70} that $T_1,T_2,T_3$
  do satisfy the scalar version of von Neumann's
  inequality.

  This example also demonstrates that Lemma \ref{lem:weights_unit} and \cite[Corollary 6]{JL79}
  do not hold in general without the assumption that the weights are non-zero. Indeed, it is not hard
  to see that if $a_{i,j} = 1$ for all $i \neq j$ in the example above, then $T$ does dilate to a commuting tuple
  of unitaries.

\begin{ack}
  The author would like to thank his advisor, Ken Davidson, for his advice and support.
\end{ack}

\bibliographystyle{amsplain}
\bibliography{literature}

\end{document}